\documentclass{amsart}

\usepackage{amsthm,amsmath,amssymb,amsfonts,epsfig}
\usepackage{hyperref}
\usepackage{stmaryrd}
\usepackage{epstopdf}
\usepackage{eucal}

\numberwithin{equation}{section}

\theoremstyle{plain}
\newtheorem{theorem}{Theorem}[section]

\newtheorem{lemma}[theorem]{Lemma}

\newtheorem{corollary}[theorem]{Corollary}
\theoremstyle{definition}
\newtheorem{definition}[theorem]{Definition}
\theoremstyle{remark}
\newtheorem{remark}[theorem]{Remark}
\newtheorem{example}[theorem]{Example}
\newcounter{counter_a}

\newcounter{counter_b}

\newcommand{\cL}{\mathcal{L}}

\newcommand{\Dom}{{\rm Dom}}

\newcommand{\Span}{{\rm span}}

\newcommand{\dist}{{\rm dist}}

\newcommand{\ess}{\mathrm{ess}}
\newcommand{\rank}{\mathrm{Rank}}
\newcommand{\range}{\mathrm{Range}}
\newcommand{\dis}{\mathrm{dis}}

\newcommand\void[1]{}
\begin{document}
\title[A new approach to spectral approximation]{A new approach to spectral approximation}
\author[M. Strauss]{Michael Strauss}
\begin{abstract}
A new technique for approximating eigenvalues and eigenvectors of a self-adjoint operator is presented. The method does not incur spectral pollution, uses trial spaces from the form domain, has a self-adjoint algorithm, and exhibits superconvergence.

\vspace{4pt}\noindent\emph{Keywords:} Eigenvalue problem, spectral pollution, Galerkin method, finite section method, spectral approximation.

\vspace{4pt}\noindent\emph{2010 Mathematics Subject Classification:} 47A58, 47A75.
\end{abstract}

\maketitle
\thispagestyle{empty}
\section{Introduction}
We shall introduce a new technique for computing eigenvalues and eigenvectors of an arbitrary semi-bounded self-adjoint operator. The method can reliably locate those eigenvalues which, due to \emph{spectral pollution}, are not successfully approximated by a direct application of the Galerkin (finite section) method; see for example \cite{agm,boff,boff2,bobole,bost,dasu,DP,lesh,rapp}. The technique is very easy to apply,  uses trial spaces from the form domain, has a self-adjoint algorithm, and exhibits the superconvergence of the Galerkin method. The only comparable technique is the recently developed perturbation method; see \cite{mar1,mar2,me,me4}. However, the latter involves perturbing eigenvalues off the real line and then approximating them, consequently, the method is compromised by having a non-self-adjoint algorithm.

There have been two approaches to locating eigenvalues and eigenvectors when spectral pollution precludes approximation with the Galerkin method. Firstly, certain specialised techniques have been proposed, with each being applicable only to a particular class of differential operator; see for example \cite{agm,bobole,doesse,mar1,mar2,stwe1,stwe2,tes}. Secondly, general techniques which are applicable to self-adjoint or even to arbitrary operators; see for example \cite{dav,DP,ah1,ah2,ah3,lesh,shar,me,me4,zim}. The last two decades has seen an enormous effort directed at general techniques for self-adjoint operators. This effort led to the development of \emph{quadratic} methods, so-called because of their reliance on truncations of the square of the operator. Although pollution-free, these techniques have significant drawbacks. They require trial spaces from the operator domain, rather than the preferred form domain. The latter being far more convenient, for example, it will usually contain the FEM space of piecewise linear trial functions. Furthermore, only the form domain is typically supported by FEM software, consequently, applying a quadratic method can be arduous. A more serious problem concerns convergence rates. Quadratic methods will typically converge to eigenvalues very slowly when compared to the superconvergence of the Galerkin method; see for example \cite[Section 6]{me3}, \cite[examples 3.5 \& 4.3]{bost} and \cite[Example 8]{me4}. The quadratic method which has received the most attention is the \emph{second order relative spectrum}; see for example \cite{bo,bb,bole,bost,dav,lesh,shar,shar2,me2,me3}. This technique has the additional drawback of a non-self-adjoint algorithm. 

In Section 2, we establish the strong convergence, with respect to the appropriate norm, of Galerkin spectral projections which are associated to a given interval. In Section 3, we give and justify a new definition of spectral pollution. Our main result is Theorem \ref{espaces} which shows that by applying the Galerkin method to an auxiliary spectral problem, with respect to certain Galerkin spectral subspaces, we achieve pollution-free spectral approximation for our underlying self-adjoint operator. In Section 5, we apply our new technique to self-adjoint operators whose eigenvalues are not located by a direct application of the Galerkin method.

\section{Galerkin spectral projections}\label{galsub}

Throughout, $A$, denotes a bounded or semi-bounded (from below) self-adjoint operator acting on a Hilbert space $\mathcal{H}$ with inner product $\langle\cdot,\cdot\rangle$. The corresponding quadratic form we denote by $\frak{a}$. $\mathcal{H}_\frak{a}$ will be the Hilbert space with the inner-product
\[
\langle u,v\rangle_{\frak{a}}:=\frak{a}[u,v] - (m-1)\langle u,v\rangle\quad\forall u,v\in\Dom(\frak{a})\quad\textrm{where}\quad m:=\min\sigma(A)
\]
and norm
\begin{align*}
\Vert u\Vert_{\frak{a}} &=(\frak{a}[u,u] - (m-1)\langle u,u\rangle)^{\frac{1}{2}}=\Vert(A-m+1)^{\frac{1}{2}}u\Vert.
\end{align*}
We always assume that $a,b\in\mathbb{R}$, $a<b$, $a,b\in\rho(A)$ and
$\sigma_{\ess}(A)\cap[a,b]=\varnothing$.  We set $\Delta=[a,b]$, and it is our goal to approximate
\[\{\lambda_1,\dots,\lambda_d\}=\sigma(A)\cap\Delta\subset\sigma_{\dis}(A)\]
where the eigenvalues are repeated according to multiplicity. The eigenspace associate to $\sigma(A)\cap\Delta$ is denote by $\mathcal{L}(\Delta)$. Let $E$ be the orthogonal projection from $\mathcal{H}$ onto $\mathcal{L}(\Delta)$, i.e., if $\Gamma$ is the circle with centre $(a+b)/2$ and radius $(b-a)/2$, then
\[
E := -\frac{1}{2\pi i}\int_\Gamma (A-\zeta)^{-1}~d\zeta\quad\text{and}\quad\rank(E) = d<\infty.
\]

We shall use the following notions of the gap/distance between two subspaces $\mathcal{L}$ and $\mathcal{M}$ of $\mathcal{H}$:
\[
\hat{\delta}(\mathcal{L},\mathcal{M}) = \max\{\delta(\mathcal{L},\mathcal{M}),\delta(\mathcal{M},\mathcal{L})\}
\quad\textrm{where}\quad
\delta(\mathcal{L},\mathcal{M}) = \sup_{u\in\mathcal{L},\Vert u\Vert=1}\dist(u,\mathcal{M});
\]
see \cite[Section IV.2.1]{katopert} for further details. We shall write $\hat{\delta}_{\frak{a}}$ and $\delta_{\frak{a}}$  to indicate the gap/distance between subspaces of the Hilbert space $\mathcal{H}_\frak{a}$.

It will always be assumed that $(\cL_n)$ is a sequence of finite-dimensional trial spaces, each being contained in $\mathcal{H}_{\frak{a}}$. We will always assume that:
\[
\forall u\in\Dom(\frak{a})\quad\exists u_n\in\cL_n:\quad\Vert u-u_n\Vert_{\frak{a}}\to0.
\]
This is the standard condition imposed upon trial spaces when employing the Galerkin method to locate eigenvalues.
The Galerkin eigenvalues of $A$ with respect to $\mathcal{L}_n$, denoted $\sigma(A,\mathcal{L}_n)$, consists of those $\mu\in\mathbb{R}$ for which there exists a non-zero vector $u\in\mathcal{L}_n$ with
\[
\frak{a}[u,v] = \mu\langle u,v\rangle\quad\forall v\in\mathcal{L}_n.
\]
Let $P_n$ be the orthogonal projection from $\mathcal{H}$ onto $\mathcal{L}_n$. Associated to the restriction of $\frak{a}$, to the subspace $\mathcal{L}_n$, is a self-adjoint operator $A_n$ which acts on $\mathcal{L}_n$ and satisfies
\[
\langle A_nu,v\rangle = \frak{a}[u,v]\quad\forall u,v\in\mathcal{L}_n.
\]
Obviously, if $\mathcal{L}_n\subset\Dom(A)$, which is always the case when $A$ is bounded, then $A_n=P_nA|_{\mathcal{L}_n}$. Let $E_n$ the spectral measure associated to $A_n$. Then $E_n(\Delta)$ is the orthogonal projection from $\mathcal{L}_n$ onto the eigenspace corresponding to $\sigma(A,\mathcal{L}_n)\cap\Delta$ and we denote this eigenspace by $\mathcal{L}_n(\Delta)$. Let $Q_n$ be the orthogonal projection from $\mathcal{H}$ onto $\mathcal{L}_n(\Delta)$, i.e., $Q_n=E_n(\Delta)P_n$.

\begin{lemma}\label{l1}
If $A$ is bounded, then $\Vert(Q_n-E)u\Vert\to0$ for all $u\in\mathcal{H}$.
\end{lemma}
\begin{proof}
The sequence of operators $(A_nP_n)$ converges strongly to $A$. It follows that $E_n((-\infty,a))P_n$ and $E_n((b,\infty))P_n$ converge strongly to $E((-\infty,a))$ and $E((b,\infty))$, respectively; see for example \cite[Corollary VIII.1.6]{katopert}. Let $u\in\mathcal{H}$, then
\begin{align*}
Q_nu = E_n(\Delta)P_nu =  P_nu - E_n((-\infty,a))P_nu - E_n((b,\infty))P_nu\to Eu.
\end{align*}
\end{proof}

We denote by $\tilde{\mathcal{L}}_n$ the range of the restriction of $(A-m+1)^{\frac{1}{2}}$ to $\mathcal{L}_n$. The orthogonal projection from $\mathcal{H}$ onto $\tilde{\mathcal{L}}_n$ we denote by $\tilde{P}_n$.

\begin{lemma}\label{tp}
$\Vert(\tilde{P}_n -I)u\Vert\to0$ for all $u\in\mathcal{H}$.
\end{lemma}
\begin{proof}
Let $u\in\mathcal{H}$ and $v=(A-m+1)^{-\frac{1}{2}}u$. There exist vectors $v_n\in\mathcal{L}_n$ with $\Vert v_n-v\Vert_{\frak{a}}\to 0$. Then
\begin{align*}
\Vert(\tilde{P}_n -I)u\Vert&\le\Vert(A-m+1)^{\frac{1}{2}}v_n-u\Vert=\Vert(A-m+1)^{\frac{1}{2}}(v_n-v)\Vert = \Vert v_n-v\Vert_{\frak{a}}.
\end{align*}
\end{proof}

\begin{lemma}\label{l2}
If $A$ is semi-bounded, then  $\Vert(Q_n -E)u\Vert_{\frak{a}}\to 0$ for all $u\in\mathcal{H}_{\frak{a}}$.
\end{lemma}
\begin{proof}
Let $u_{n,1},\dots,u_{n,d_n}$ be orthonormal eigenvectors corresponding to the intersection $\sigma(A,\cL_n)\cap\Delta$. Then $\langle u_{n,i},u_{n,j}\rangle=\delta_{ij}$ and
\[\frak{a}[u_{n,j},v]=\mu_{n,j}\langle u_{n,j},v\rangle\quad\forall v\in\cL_n\quad\textrm{where}\quad\mu_{n,j}\in\Delta.\]
For each $v\in\cL_n$, we set $\tilde{v}=(A-m+1)^{\frac{1}{2}}v\in\tilde{\cL}_n$. Then
\[
\langle(A-m+1)^{-1}\tilde{u}_{n,j},\tilde{v}\rangle=\frac{1}{\mu_{n,j}-m+1}\langle\tilde{u}_{n,j},\tilde{v}\rangle\quad\forall\tilde{v}\in\tilde{\cL}_n
\]
where
\[
\frac{1}{\mu_{n,j}-m+1}\in\tilde{\Delta}:=\left[\frac{1}{b-m+1},\frac{1}{a-m+1}\right].
\]
Note that $E$ is the spectral projection associated to the self-adjoint operator $(A-m+1)^{-1}$ and the interval $\tilde{\Delta}$. Furthermore, the set
\begin{equation}\label{invvecs}
\left\{\frac{\tilde{u}_{n,1}}{\sqrt{\mu_{n,1}-m+1}},\dots,\frac{\tilde{u}_{n,d_n}}{\sqrt{\mu_{n,d_n}-m+1}}\right\}
\end{equation}
are orthonormal eigenvectors associated to $\sigma((A-m+1)^{-1},\tilde{\cL}_n)\cap\tilde{\Delta}$. We denote by $\tilde{Q}_n$ the orthogonal projection from $\mathcal{H}$ onto
the span of the set \eqref{invvecs}. It follows, from Lemma \ref{l1} and Lemma \ref{tp}, that 
\[(\tilde{Q}_n-E)u\to 0\quad\forall u\in\mathcal{H}.
\]
Let $u\in\mathcal{H}_{\frak{a}}$
and $(A-m+1)^{\frac{1}{2}}u=v$. Then
\begin{align*}
\Vert(Q_n - E)u\Vert_{\frak{a}}
&\le\Big\Vert\sum_{j=1}^{d_n}\big\langle(A-m+1)^{-\frac{1}{2}}\tilde{P}_nv,
u_{n,j}\big\rangle u_{n,j} - Eu\Big\Vert_{\frak{a}}\\
&\quad+\Big\Vert\sum_{j=1}^{d_n}\big\langle(A-m+1)^{-\frac{1}{2}}(I-\tilde{P}_n)v,
u_{n,j}\big\rangle u_{n,j}\Big\Vert_{\frak{a}}
\end{align*}
where
\begin{align*}
\Big\Vert\sum_{j=1}^{d_n}\big\langle(A-m+1)^{-\frac{1}{2}}&\tilde{P}_nv,
u_{n,j}\big\rangle u_{n,j} - Eu\Big\Vert_{\frak{a}}\\
&=\Big\Vert\sum\big\langle\tilde{P}_nv,
(A-m+1)^{-1}\tilde{u}_{n,j}\big\rangle u_{n,j} - Eu\Big\Vert_{\frak{a}}\\
&=\Big\Vert\sum\frac{\langle\tilde{P}_nv,
\tilde{u}_{n,j}\rangle u_{n,j}}{\mu_{n,j}-m+1} - Eu\Big\Vert_{\frak{a}}\\
&=\Big\Vert\sum\frac{\langle v,\tilde{u}_{n,j}\rangle u_{n,j}}{\mu_{n,j}-m+1} - E(A-m+1)^{-\frac{1}{2}}v\Big\Vert_{\frak{a}}\\
&=\Big\Vert\sum\frac{\langle v,\tilde{u}_{n,j}\rangle\tilde{u}_{n,j}}{\Vert\tilde{u}_{n,j}\Vert^2} - Ev\Big\Vert\\
&=\Vert(\tilde{Q}_n - E)v\Vert\to 0
\end{align*}
and
\begin{align*}
&\Big\Vert\sum_{j=1}^{d_n}\big\langle(A-m+1)^{-\frac{1}{2}}(I-\tilde{P}_n)v,
u_{n,j}\big\rangle u_{n,j}\Big\Vert_{\frak{a}}\\
&\qquad\qquad=\Big\Vert\sum\big\langle(A-m+1)^{-1}(I-\tilde{P}_n)v,
\tilde{u}_{n,j}\big\rangle\tilde{u}_{n,j}\Big\Vert\\
&\qquad\qquad=\Big\Vert\sum(\mu_{n,j}-m+1)\frac{\langle(A-m+1)^{-1}(I-\tilde{P}_n)v,
\tilde{u}_{n,j}\rangle}{\Vert\tilde{u}_{n,j}\Vert^2}\tilde{u}_{n,j}\Big\Vert\\
&\qquad\qquad=\sqrt{\sum(\mu_{n,j}-m+1)^2\left\vert\frac{\langle(A-m+1)^{-1}(I-\tilde{P}_n)v,
\tilde{u}_{n,j}\rangle}{\Vert\tilde{u}_{n,j}\Vert}\right\vert^2}\\
&\qquad\qquad\le(b-m+1)\Vert\tilde{Q}_n(A-m+1)^{-1}(I-\tilde{P}_n)v\Vert\to 0.
\end{align*}
\end{proof}

\section{Spectral pollution}
When $A$ is bounded and $\Delta$ intersects the closed convex hull of $\sigma_{\ess}(A)$ or when $A$ is semi-bounded and $\min\sigma_{\ess}(A)<b$, the Galerkin method cannot usually be relied on to approximate $\sigma(A)\cap\Delta$. This is due to a phenomenon known as spectral pollution which is normally defined in terms of spurious Galerkin eigenvalues, i.e., in terms of whether or not
\begin{equation}\label{poll1}
d_{H}(\sigma(A,\mathcal{L}_n)\cap\Delta,\{\lambda_1,\dots,\lambda_d\})
\end{equation}
converges to zero, where $d_H$ is the Hausdorff distance. However, this is not entirely satisfactory. It is straightforward to construct an example with, for some subsequence $(n_j)$,
\begin{equation}\label{poll2}
\dim\mathcal{L}_{n_j}(\Delta)>\dim\mathcal{L}(\Delta)\quad\forall j\in\mathbb{N}
\end{equation}
and where we can choose  whether or not \eqref{poll1} converges to zero. In other words, the Galerkin method can fail to approximate $\mathcal{L}(\Delta)$ even if \eqref{poll1} converges to zero. 

\begin{definition}\label{poll0}
We say that \emph{spectral pollution} occurs in $\Delta$ for the sequence $(\mathcal{L}_n)$ if
$\delta(\mathcal{L}_n(\Delta),\mathcal{L}(\Delta))\nrightarrow 0$.
\end{definition}

\begin{lemma}\label{poll}
Let spectral pollution not occur in $\Delta$ for the sequence $(\mathcal{L}_n)$. If $A$ is bounded, then
\begin{equation}\label{poll6}
\hat{\delta}(\mathcal{L}_n(\Delta),\mathcal{L}(\Delta))\to 0\quad\text{and}\quad\text{d}_{H}(\sigma(A,\mathcal{L}_n)\cap\Delta,\{\lambda_1,\dots,\lambda_d\})\to 0.
\end{equation}
If $A$ is semi-bounded, then
\begin{equation}\label{poll7}
\hat{\delta}_{\frak{a}}(\mathcal{L}_n(\Delta),\mathcal{L}(\Delta))\to 0\quad\text{and}\quad\text{d}_{H}(\sigma(A,\mathcal{L}_n)\cap\Delta,\{\lambda_1,\dots,\lambda_d\})\to 0.
\end{equation}
\end{lemma}
\begin{proof}
Let $A$ be bounded. 
It follows, from Lemma \ref{l1}, that
$\delta(\mathcal{L}(\Delta),\mathcal{L}_n(\Delta))$ converges to zero. Hence
$\hat{\delta}(\mathcal{L}_n(\Delta),\mathcal{L}(\Delta))$ converges to zero,
and the right hand side of \eqref{poll6} follows.

Let $A$ be semi-bounded. It follows, from Lemma \ref{l2}, that
$\delta_{\frak{a}}(\mathcal{L}(\Delta),\mathcal{L}_n(\Delta))$ converges to zero. Hence $\hat{\delta}(\mathcal{L}(\Delta),\mathcal{L}_n(\Delta))$ comverges to zero
which implies that 
\[
\dim\mathcal{L}_n(\Delta)=\dim\mathcal{L}(\Delta)=d<\infty
\]
for all sufficiently large $n\in\mathbb{N}$. Therefore, the following formula holds
\[
\delta_{\frak{a}}(\mathcal{L}_n(\Delta),\mathcal{L}(\Delta)) \le 
\frac{\delta_{\frak{a}}(\mathcal{L}(\Delta),\mathcal{L}_n(\Delta))}{1-\delta_{\frak{a}}(\mathcal{L}(\Delta),\mathcal{L}_n(\Delta))};
\]
see \cite[Lemma 213]{kato}.
Hence
$\hat{\delta}_{\frak{a}}(\mathcal{L}_n(\Delta),\mathcal{L}(\Delta))\to 0$,
and the right hand side of \eqref{poll7} follows. 
\end{proof}

\begin{lemma}\label{poll10}
Spectral pollution occurs in $\Delta$ for the sequence $(\mathcal{L}_n)$ if and only if \eqref{poll2} holds.
\end{lemma}
\begin{proof}
If \eqref{poll2} holds then there exist non-zero vectors $u_{n_j}\in\mathcal{L}_{n_j}(\Delta)$ with $u_{n_j}\perp\mathcal{L}(\Delta)$. Then
\[
\delta(\mathcal{L}_{n_j}(\Delta),\mathcal{L}(\Delta))=1\quad\forall j\in\mathbb{N}.
\]
Hence spectral pollution occurs in $\Delta$ for the sequence $(\mathcal{L}_n)$.

If spectral pollution occurs in $\Delta$ for the sequence $(\mathcal{L}_n)$, then 
\begin{equation}\label{poll9}
\delta(\mathcal{L}_n(\Delta),\mathcal{L}(\Delta))\nrightarrow 0.
\end{equation}
We have, by Lemma \ref{l1} if $A$ is bounded and by Lemma \ref{l2} if $A$ is semi-bounded, 
$\delta(\mathcal{L}(\Delta),\mathcal{L}_n(\Delta))\to 0$.
Therefore 
\[\dim\mathcal{L}_n(\Delta)\ge\dim\mathcal{L}(\Delta)\quad\text{for all sufficiently large}\quad n\in\mathbb{N}.\]
If 
\[\dim\mathcal{L}_n(\Delta)=\dim\mathcal{L}(\Delta)=d<\infty\quad\text{for all sufficiently large}\quad n\in\mathbb{N},\]
then
\[
\delta(\mathcal{L}_n(\Delta),\mathcal{L}(\Delta)) \le 
\frac{\delta(\mathcal{L}(\Delta),\mathcal{L}_n(\Delta))}{1-\delta(\mathcal{L}(\Delta),\mathcal{L}_n(\Delta))}\to 0,
\]
which contradicts \eqref{poll9}. We deduce that \eqref{poll2} holds.
\end{proof}

Our definition of spectral pollution fits naturally within the well known theory of spectral approximation using the Galerkin method. For example, if $A$ is bounded and spectral pollution does not occur in $\Delta$, then the sequence $(A_nP_n)$ is called a \emph{strongly stable} approximation of $A$ at $\sigma(A)\cap\Delta$. For strongly stable approximations the Galerkin method is extremely effective and well understood; see for example the excellent monograph \cite{chat}.

\section{Approximation of $\sigma(A)\cap\Delta$}

Throughout this section, $\mathcal{L}$ will denote a fixed and finite-dimensional subspace of $\mathcal{H}$. The orthogonal projection from $\mathcal{H}$ onto $\mathcal{L}$ is denoted by $P$.

\begin{lemma}\label{prelims}
The operator $EPE$ is self-adjoint and non-negative,
$\sigma(EPE)\subset[0,1]$, $\sigma_{\ess}(EPE)=\{0\}$
and $\sigma_{\dis}(EPE)$ consists of $\rank(PE)$ non-zero eigenvalues counted according to multiplicity. Furthermore,
if $EPEu = \mu u$ where $\mu\ne 0$, then $u\in\mathcal{L}(\Delta)$.
\end{lemma}
\begin{proof}
Evidently, $EPE$ is a self-adjoint operator and
\begin{equation}\label{triv}
\langle EPEv,v\rangle = \langle PEv,Ev\rangle = \Vert PEv\Vert^2 \le \Vert v\Vert^2\quad\forall v\in\mathcal{H}.
\end{equation}
The second and third assertions follow. That $\sigma_{\ess}(EPE)=\{0\}$ follows from the fact that $EPE$ is finite rank. 
From \eqref{triv} we deduce that $EPEv=0$ iff $PEv=0$ and therefore $\rank(EPE)=\rank(PE)$ and the fifth assertion follows. Finally, if 
$EPEu = \mu u$, then $EPEu = \mu Eu$ and hence $u\in\mathcal{L}(\Delta)$.  
\end{proof}

\void{
\begin{lemma}\label{l1b}
If $A$ is bounded, then
\[\Vert(I-B_n)E(\Delta)\Vert=\mathcal{O}\big(\delta(\cL(\Delta),\cL_n)\big)\quad\textrm{and}\quad\delta\big(\cL(\Delta),\cL_n(\Delta)\big)=\mathcal{O}\big(\delta(\cL(\Delta),\cL_n)\big).\]
\end{lemma}
\begin{proof}
Let $(A-\lambda_j)u = 0$ with $\Vert u\Vert=1$ and $\lambda_j\in\Delta$. Set $u_n=P_nu$, then
\begin{align*}
\Vert(A_n-\lambda_j)u_n\Vert 
=\Vert(P_nA-\lambda_j)u_n - P_n(A-\lambda_j) u\Vert\le\Vert A\Vert\delta(\cL(\Delta),\cL_n),
\end{align*}
and
\begin{align*}
\Vert(I-B_n)u_n\Vert^2 &\le\int_{\mathbb{R}\backslash\Delta}\frac{\vert\mu-\lambda_j\vert^2}{\dist[\lambda_j,\{a,b\}]^2}~d\langle (E_n)_\mu u_n,u_n\rangle\\
&\le\frac{1}{\dist[\lambda_j,\{a,b\}]^2}\int_{\mathbb{R}}\vert\mu-\lambda_j\vert^2~d\langle (E_n)_\mu u_n,u_n\rangle\\
&\le\frac{\Vert A\Vert^2\delta(\cL(\Delta),\cL_n)^2}{\dist(\lambda_j,\{a,b\})^2}.
\end{align*}
Therefore
\begin{align*}
\Vert(I- B_n)u\Vert &\le \Vert(I-B_n)\Vert + \Vert(I-B_n)(u-u_n)\Vert\\
&\le\left(\frac{\Vert A\Vert}{\dist(\lambda_j,\{a,b\})} + 1\right)\delta(\cL(\Delta),\cL_n),
\end{align*}
from which both assertions follow.
\end{proof}
}

\void{
\begin{definition}\label{def1}
\begin{align*}
\sigma(P,\mathcal{L}_\infty(\Delta)):=\big\{\mu\in\mathbb{R}: \text{ there exists a }&\text{sequence }(\mu_n)_{n\in\mathbb{N}}\text{ with }\\
&\vert\mu-\mu_n\vert\to0 
\text{ and }\mu_n\in\sigma(P,\mathcal{L}_n(\Delta))\big\}.
\end{align*}
\end{definition}
}

\begin{lemma}\label{thm1}
\begin{equation}\label{ds1}
\max_{\mu\in\sigma(EPE)\backslash\{0\}}\dist\big(\mu,\sigma(P,\mathcal{L}_n(\Delta))\big)\to 0
\end{equation}
and \begin{equation}\label{ds2}
\max_{\mu\in\sigma(P,\mathcal{L}_n(\Delta))}\dist\big(\mu,\sigma(EPE)\big)\to 0.
\end{equation}
\end{lemma}
\begin{proof}
Let $\mu\in\sigma(EPE)\backslash\{0\}$. By Lemma \ref{prelims}, $\mu\in\sigma_{\dis}(EPE)$
and the corresponding eigenspace is a subset of $\mathcal{L}(\Delta)$. Let $u\in\mathcal{L}(\Delta)$ with
$EPEu = \mu u$. It follows, from Lemma \ref{l1} if $A$ is bounded and from Lemma \ref{l2} if $A$ is semi-bounded, that
\[
Q_nu\to u\quad\text{and}\quad Q_nPQ_nu\to EPEu = \mu u,
\]
hence
$Q_nPQ_nu-\mu Q_nu\to 0$. Then \eqref{ds1} follows from the fact that $Q_nP|_{\mathcal{L}_n(\Delta)}$ is a self-adjoint operator and $EPE$ is finite rank.

Let $(\mu_n)$ be a sequence with $\mu_n\in\sigma(P,\mathcal{L}_n(\Delta))$.
There are normalised vectors $u_n\in\mathcal{L}_n(\Delta)$ with
$Q_nPu_n = \mu_nu_n$.
From  Lemma \ref{l1} if $A$ is bounded and from Lemma \ref{l2} if $A$ is semi-bounded, and the fact that $P$ is finite-rank, we deduce that 
\[
\Vert Q_nPQ_n - EPE\Vert\to 0.\]
Hence
\begin{align*}
EPEu_n - \mu_n u_n &=(EPE-Q_nP)u_n  +Q_nPu_n  - \mu_n u_n\\
&=  (EPE-Q_nPQ_n)u_n \to 0.
\end{align*}
Then \eqref{ds2} follows from the fact that $EPE$ is a self-adjoint operator and  \[\rank(Q_nP|_{\mathcal{L}_n(\Delta)})\le\rank(P)<\infty.\]
\end{proof}

\void{
In Definition \ref{def1}, we could have defined the limit set to be those $\mu\in\mathbb{R}$ for which $\mu_{n_j}\to\mu$ and $\mu_{n_j}\in\sigma(P,\mathcal{L}_{n_j}(\Delta))$ for some subsequence $n_j$. This appears to be a more general definition giving rise to a possibly larger set. However, the following lemma shows that this is not the case.
\begin{lemma}
Let $(n_j)_{j\in\mathbb{N}}$ be an increasing sequence of natural numbers. If $\mu_{n_j}\in\sigma(P,\mathcal{L}_{n_j}(\Delta))$ for each $j\in\mathbb{N}$ and $\mu_{n_j}\to\mu$, then $\mu\in\sigma(P,\mathcal{L}_\infty(\Delta))$.
\end{lemma}
\begin{proof}
We have a sequence of normalised vectors $(v_{n_j})_{j\in\mathbb{N}}$ with $v_{n_j}\in\mathcal{L}_{n_j}(\Delta)$ and
\[
Q_{n_j}Pv_{n_j} = \mu_{n_j}v_{n_j}\quad\text{where}\quad\mu_{n_j}\to\mu.
\]
Then, similarly to the proof of Lemma \ref{thm1}, we deduce that
\begin{align*}
EPEv_{n_j} - \mu v_{n_j} \to 0
\end{align*}
and therefore that $\mu\in\sigma(EPE)$. Then $\mu\in\sigma(P,\mathcal{L}_\infty(\Delta))$ by Lemma \ref{thm1}.
\end{proof}
}

\begin{theorem}\label{espaces}
Let $\delta(\mathcal{L}(\Delta),\mathcal{L})<1$. Then
\begin{equation}\label{approxed00}
0<\min\sigma_{\dis}(EPE)=:\gamma.
\end{equation}
Let $\mathcal{M}_n$ be the spectral subspace associated to $Q_nP|_{\mathcal{L}_n(\Delta)}$ and the interval $[\gamma/2,1]$. Then
\begin{equation}\label{approxed1}
\hat{\delta}(\mathcal{M}_n,\mathcal{L}(\Delta))\to 0.
\end{equation}
If $A$ is bounded, then
\begin{equation}\label{approxed2}
d_{H}(\sigma(A,\mathcal{M}_n),\{\lambda_1,\dots,\lambda_d\})\to 0.
\end{equation}
If $A$ is semi-bounded, then
\begin{equation}\label{approxed3}
\hat{\delta}_{\frak{a}}(\mathcal{M}_n,\mathcal{L}(\Delta))\to 0\quad\text{and}\quad d_{H}(\sigma(A,\mathcal{M}_n),\{\lambda_1,\dots,\lambda_d\})\to 0.
\end{equation}
\end{theorem}
\begin{proof}
Since $\delta(\mathcal{L}(\Delta),\mathcal{L})<1$, we have
\[
\Vert(I-P)u\Vert<\Vert u\Vert\quad\forall u\in\mathcal{L}(\Delta)\backslash\{0\}.
\]
We deduce that $\rank(PE)=d$. Then, by Lemma \ref{prelims}, $\sigma_{\dis}(EPE)$ consists of $d$ (repeated) positive eigenvalues from which \eqref{approxed00} follows. Furthermore, all corresponding eigenvectors belong to $\mathcal{L}(\Delta)$. Hence, $\mathcal{L}(\Delta)$ is the spectral subspace corresponding to $EPE$ and the interval $[\gamma,1]$. Let $\Gamma$ be the circle with centre $(1+\gamma)/2$ and radius $1/2$, then $\Vert EPEu-\zeta u\Vert\ge\gamma/2\Vert u\Vert$ for all $\zeta\in\Gamma$ and $u\in\mathcal{H}$. Using Lemma \ref{thm1} and the self-adjointness of the operators $Q_nPQ_n$,
\[
\Gamma\subset\rho(Q_nPQ_n)\quad\text{and}\quad\max_{\zeta\in\Gamma}\Vert(Q_nPQ_n - \zeta)^{-1}\Vert\le \frac{4}{\gamma}
\]
for all sufficiently large $n\in\mathbb{N}$. Let $F_n$ be the orthogonal projection from $\mathcal{H}$ onto $\mathcal{M}_n$, then
\begin{align*}
\Vert E - F_n\Vert &= \left\Vert-\frac{1}{2\pi i}\int_\Gamma (EPE - \zeta)^{-1} -  (Q_nPQ_n - \zeta)^{-1} ~d\zeta\right\Vert\\
&\le\frac{4\Vert Q_nPQ_n - EPE\Vert}{\gamma^2}\to 0,
\end{align*} 
and \eqref{approxed1} follows. If $A$ is bounded, then  \eqref{approxed2} follows immediately from  \eqref{approxed1}. Let $A$ be semi-bounded and
denote by $\hat{A}_n$ the restriction of $A_n$ to $\mathcal{L}_n(\Delta)$. Then $\hat{A}_n$ is a self-adjoint operator with $\Vert\hat{A}_n\Vert\le b$ for all $n\in\mathbb{N}$. Let $v\in\mathcal{L}(\Delta)$ and $\Vert v\Vert_{\frak{a}}=1$. It follows, from \eqref{approxed1}, that there exists a sequence of vectors  $(v_n)$ with
\[
\Vert v_n - v\Vert\to 0\quad\text{and}\quad v_n\in\mathcal{M}_n.
\]
Then
\[
\Vert v_n - v\Vert_{\frak{a}} \le \Vert v_n - Q_nv\Vert_{\frak{a}} + \Vert Q_nv- v\Vert_{\frak{a}}
\]
where $\Vert Q_nv- v\Vert_{\frak{a}}\to 0$ by Lemma \ref{l2}, and
\begin{align*}
\Vert v_n - Q_nv\Vert_{\frak{a}}^2 &= \frak{a}[v_n-Q_nv] + (1-m)\Vert v_n-Q_nv\Vert^2\\
&=\langle A_n(v_n-Q_nv),v_n-Q_nv\rangle+ (1-m)\Vert v_n-Q_nv\Vert^2\\
&=\langle \hat{A}_n(v_n-Q_nv),v_n-Q_nv\rangle+ (1-m)\Vert v_n-Q_nv\Vert^2\\
&\le\vert\langle \hat{A}_n(v_n-Q_nv),v_n-Q_nv\rangle\vert+ (1-m)\Vert v_n-Q_nv\Vert^2\\
&\le\Vert\hat{A}_n\Vert\Vert v_n-Q_nv)\Vert^2+ (1-m)\Vert v_n-Q_nv\Vert^2\\
&\le(b+1-m)\Vert v_n-Q_nv\Vert^2\to 0.
\end{align*}
Since $\dim\mathcal{L}(\Delta)=d<\infty$, we deduce that $\delta_{\frak{a}}(\mathcal{L}(\Delta),\mathcal{M}_n)\to 0$.
It follows, from \eqref{approxed1}, that $\dim\mathcal{M}_n=\dim\mathcal{L}(\Delta)$ for all sufficiently large $n\in\mathbb{N}$, and therefore that
\[
\delta_{\frak{a}}(\mathcal{M}_n,\mathcal{L}(\Delta)) \le \frac{\delta_{\frak{a}}(\mathcal{L}(\Delta),\mathcal{M}_n)}{1-\delta_{\frak{a}}(\mathcal{L}(\Delta),\mathcal{M}_n)}.
\]
Both assertions in \eqref{approxed3} follow.
\end{proof}

\begin{remark}
The condition $\delta(\mathcal{L}(\Delta),\mathcal{L})<1$ is extremely mild. We do not require $\mathcal{L}$ to approximate $\mathcal{L}(\Delta)$ in any meaningful sense, we require only that there be no non-zero vector $u\in\mathcal{L}(\Delta)$ with $u\perp\mathcal{L}$. Put another way, we require $\range(E|_{\mathcal{L}}) = \mathcal{L}(\Delta)$.
\end{remark}

The following corollary demonstrates a nice feature of $\sigma(A,\mathcal{M}_n)$: if spectral pollution does not occur in $\Delta$ for the sequence $(\mathcal{L}_n)$, then approximating $\sigma(A)\cap\Delta$ with $\sigma(A,\mathcal{M}_n)$ is equivalent to using the Galerkin method. 

\begin{corollary}\label{espaces3}
Let $\delta(\mathcal{L}(\Delta),\mathcal{L})<1$ and $\delta(\mathcal{L}_n(\Delta),\mathcal{L}(\Delta))\to 0$. Then 
\begin{equation}\label{approxed4}
\sigma(A,\mathcal{M}_n) = \sigma(A,\mathcal{L}_n)\cap\Delta\quad\text{for.all sufficiently large}\quad n\in\mathbb{N}.
\end{equation}
\end{corollary}
\begin{proof}
It follows, from  \eqref{approxed1}, that  $\dim(\mathcal{M}_n)=d$ for all sufficiently large $n\in\mathbb{N}$, and hence $\mathcal{M}_n=\mathcal{L}_n(\Delta)$.
\end{proof}

The following corollary shows that the condition $\delta(\mathcal{L}(\Delta),\mathcal{L})<1$ can be relaxed when $\Delta$ contains only one eigenvalue. Knowing the latter is  essential if applying the quadratic methods described in \cite{DP,zim}, it is also essential for obtaining an eigenvalue enclosure using the second order relative spectrum; see for example \cite[Remark 2.3]{me2}.

\begin{corollary}\label{espaces4}
Let $\Delta\cap\sigma(A)=\lambda$ and $\rank(PE)\ne 0$. Then 
\begin{equation}\label{approxed5a}
0<\min\sigma_{\dis}(EPE)=:\gamma.
\end{equation}
Let $\mathcal{M}$ be the eigenspace corresponding to $EPE$ and the interval $[\gamma,1]$. Then
\begin{equation}\label{approxed5}
\dim\mathcal{M}=\rank(PE),\quad\mathcal{M}\subset\mathcal{L}(\Delta)\quad\text{and}\quad \hat{\delta}(\mathcal{M}_n,\mathcal{M})\to 0.
\end{equation}
If $A$ is bounded, then
\begin{equation}\label{approxed6}
d_{H}(\sigma(A,\mathcal{M}_n),\lambda)\to 0.
\end{equation}
If $A$ is semi-bounded, then
\begin{equation}\label{approxed7}
\hat{\delta}_{\frak{a}}(\mathcal{M}_n,\mathcal{M})\to 0\quad\text{and}\quad d_{H}(\sigma(A,\mathcal{M}_n),\lambda)\to 0.
\end{equation}
\end{corollary}
\begin{proof}
The assertion \eqref{approxed5a} and the first two assertions in \eqref{approxed5} follow immediately from Lemma \ref{prelims}. Let $F$ and $F_n$ be the orthogonal projections from $\mathcal{H}$ onto $\mathcal{M}$ and $\mathcal{M}_n$, respectively. 
Then arguing similarly to the proof of Theorem \ref{espaces}
\begin{align*}
\Vert F - F_n\Vert &= \left\Vert-\frac{1}{2\pi i}\int_\Gamma (EPE - \zeta)^{-1} -  (Q_nPQ_n - \zeta)^{-1} ~d\zeta\right\Vert\\
&\le\frac{4\Vert Q_nPQ_n - EPE\Vert}{\gamma^2}\to 0.
\end{align*} 
The right hand side of \eqref{approxed5} follows. If $A$ is bounded then  \eqref{approxed6} follows immediately from  \eqref{approxed5}. Let $A$ be semi-bounded. Let $v\in\mathcal{M}$ and $\Vert v\Vert_{\frak{a}}=1$. It follows, from \eqref{approxed5}, that there exists a sequence of vectors  $(v_n)$ with
\[
\Vert v_n - v\Vert\to 0\quad\text{and}\quad v_n\in\mathcal{M}_n.
\]
That $\Vert v_n - v\Vert_{\frak{a}}\to 0$ may be proved in precisely the same way as in the proof of Theorem \ref{espaces}. Both assertions in \eqref{approxed7} follow.
\end{proof}

Our final corollary shows that a natural choice for $\mathcal{L}$ will be one of our trial spaces. Choosing such an $\mathcal{L}$ will also be extremely convenient.

\begin{corollary}
Let $\mathcal{L}=\mathcal{L}_k$. There exists an $N\in\mathbb{N}$ such that 
\[
\hat{\delta}(\mathcal{M}_n,\mathcal{L}(\Delta))\to 0\quad\forall k\ge N.
\]
If $A$ is bounded and $k\ge N$, then
\[
d_{H}(\sigma(A,\mathcal{M}_n),\{\lambda_1,\dots,\lambda_d\})\to 0.
\]
If $A$ is semi-bounded and $k\ge N$, then
\[
\hat{\delta}_{\frak{a}}(\mathcal{M}_n,\mathcal{L}(\Delta))\to 0\quad\text{and}\quad d_{H}(\sigma(A,\mathcal{M}_n),\{\lambda_1,\dots,\lambda_d\})\to 0.
\]
\end{corollary}
\begin{proof}
The existence of such an $N$ is ensured by the fact that $\delta(\mathcal{L}(\Delta),\mathcal{L}_n)$ converges to zero. Hence, all assertions follow from Theorem \ref{espaces}.
\end{proof}

\section{Examples}

The results of the previous section present us with a new approach to approximating eigenspaces and eigenvalues. Before looking at some examples, let us briefly discuss the procedure we should follow. 

First, suppose that we know $d=\dim\mathcal{L}(\Delta)$. Then we should choose an $\mathcal{L}$ with $\dim\mathcal{L}\ge d$. An obvious choice for $\mathcal{L}$ being one of our trial spaces. With $\mathcal{L}$ chosen, we compute $\sigma(P,\mathcal{L}_n(\Delta))$ for increasing values of $n$. If $\sigma(P,\mathcal{L}_n(\Delta))$ appears to converge to $d$ non-zero eigenvalues, then we may approximate $\Delta\cap\sigma(A)$ with $\sigma(A,\mathcal{M}_n)$. If instead, $\sigma(P,\mathcal{L}_n(\Delta))$ appears to converge to less than $d$ non-zero eigenvalues, then this suggests that $\delta(\mathcal{L}(\Delta),\mathcal{L})\approx 1$. We should choose a new $\mathcal{L}$ with the natural choice being a higher dimensional trial space, and repeat.

Secondly, suppose that we do not know $d=\dim\mathcal{L}(\Delta)$. Choose a low dimensional $\mathcal{L}$, say $\dim\mathcal{L}=s$. Then compute $\sigma(P,\mathcal{L}_n(\Delta))$ for increasing values of $n$. If $\sigma(P,\mathcal{L}_n(\Delta))$ appears to converge to $r<s$ non-zero eigenvalues, then we may approximate $\Delta\cap\sigma(A)$ with $\sigma(A,\mathcal{M}_n)$.  If instead, $\sigma(P,\mathcal{L}_n(\Delta))$ appears to converge to $s$ non-zero eigenvalues, then this suggests that $\dim\mathcal{L}(\Delta)\ge\dim\mathcal{L}$. We should choose a new $\mathcal{L}$ with the natural choice being a higher dimensional trial space, and repeat. We might also increase the dimension of $\mathcal{L}$ and check that $\sigma(P,\mathcal{L}_n(\Delta))$ still converges to $r$ non-zero eigenvalues.

\begin{example}\label{ex1}
Let $\mathcal{H}=L^2(-\pi,\pi)$ and $v_k = (2\pi)^{-\frac{1}{2}}e^{-ikx}$ for $k\in\mathbb{Z}$. For each $u\in\mathcal{H}$,
\begin{align*}
Au = a(x)u + 10\langle u,v_0\rangle v_0\quad\textrm{where}\quad a(x)=\begin{cases}
  -2\pi - x & \text{for}\quad x\in(-\pi,0], \\
  2\pi - x & \text{for}\quad x\in(0,\pi].
  \end{cases}
\end{align*}
Then $\sigma_{\ess}(A)=[-2\pi,-\pi]\cup[\pi,2\pi]$ and $\sigma_{\dis}(A)$ consists of two simple eigenvalues 
\[\lambda_1\approx -1.64834270\quad\text{and}\quad\lambda_2\approx 11.97518502;\] see \cite[Lemma 12]{DP}. With trial spaces 
$\mathcal{L}_{2k+1}=\Span\{v_{-k},\dots,v_k\}$
we shall approximate the eigenvalue $\lambda_1$ which is located in the gap in the essential spectrum. Figure 1, shows $\sigma(A,\mathcal{L}_{n})$ for $n=17,65,257,1025$ and $4097$. Increasing with $n$, are the number of \emph{spurious} Galerkin eigenvalues in the interval $(-\pi,\pi)$. These spurious Galerkin eigenvalues obscure the approximation of $\lambda_1$; see also \cite[Example 1]{lesh}. We set $\Delta=[-\pi+0.001,\pi-0.001]$ and $\mathcal{L}=\mathcal{L}_{1}$. Then
\[
\dim\mathcal{L}(\Delta)=1,\quad\dim\mathcal{L}=1\quad\text{and}\quad\Delta\cap\sigma(A)=\{\lambda_1\}.
\]
Figure 2, shows $\sigma(P_1,\mathcal{L}_{n}(\Delta))$ and, as $n$ increases, we converge to only one non-zero eigenvalue which is approximately $0.12$. The corresponding eigenvector, by Theorem \ref{espaces}, will converge to the eigenvector corresponding to $\lambda_1$. We compare the approximation of $\lambda_1$ by $\sigma(A,\mathcal{M}_n)$ with the approximation provided by the perturbation method. For the latter, we calculate $\sigma(A+iP_{(n-1)/2},\mathcal{L}_n)$ and obtain a sequence $\mu_{n}\in\sigma(A+iP_{(n-1)/2},\mathcal{L}_n)$ with $\mu_n\to\lambda_1+i$; see \cite{me4} for further details. Figure 3, shows the distance of $\sigma(A,\mathcal{M}_n)$ to $\lambda_1$ and of $\sigma(A+iP_{(n-1)/2},\mathcal{L}_n)$ to $\lambda_1+i$. Both methods converge at about the same rate, though the actual approximation by the former is more accurate.
\begin{figure}[h!]
\centering
\includegraphics[scale=.265]{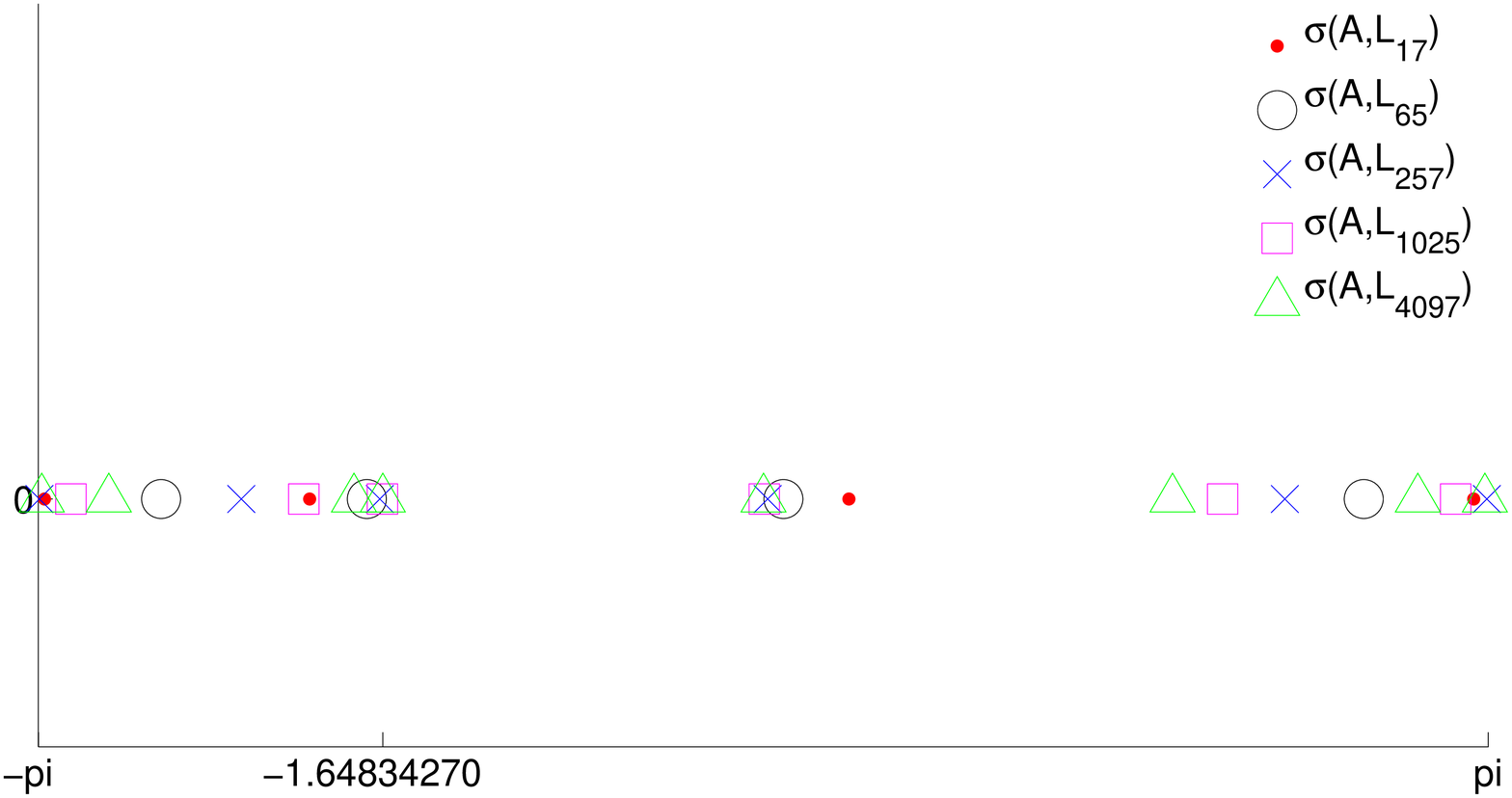}
\caption{Spurious Galerkin eigenvalues obscure the approximation of $\lambda_1$.}
\end{figure}
\begin{figure}[h!]
\centering
\includegraphics[scale=.265]{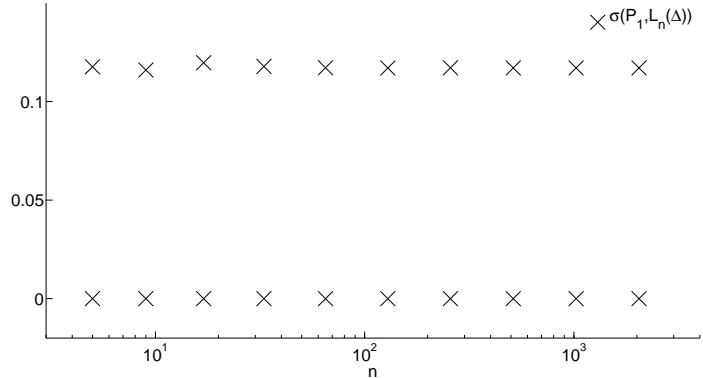}
\caption{$\sigma(P_1,\mathcal{L}_n(\Delta))$ converging to a single non-zero eigenvalue $\approx 0.12$.}
\end{figure}
\begin{figure}[h!]
\centering
\includegraphics[scale=.265]{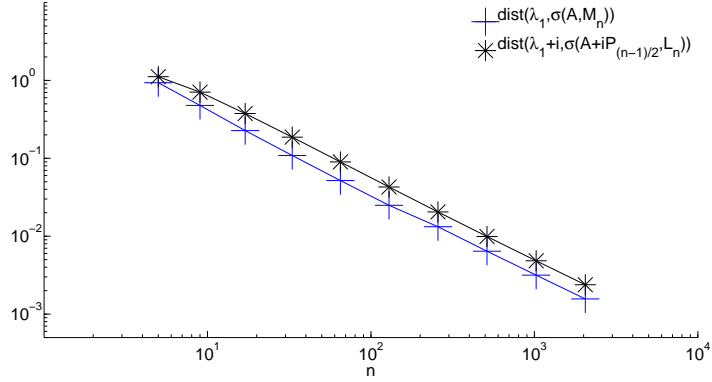}
\caption{Approximation of $\lambda_1$ with $\sigma(A+iP_{(n-1)/2},\mathcal{L}_n)$ and with $\sigma(A,\mathcal{M}_n)$.}

\end{figure}
\end{example}

\begin{example}\label{ex2}Let $\mathcal{H}=L^2(0,1)\times L^2(0,1)$ and consider the block-operator matrix
\begin{displaymath}
A_0=\small{\left(
\begin{array}{cc}
-\frac{d^2}{dx^2} & -\frac{d}{dx}\\
& \\
\frac{d}{dx} & 2
\end{array} \right)}\hspace{5pt}\textrm{with}\hspace{5pt}\Dom(A_0)=H^2(0,1)\cap H^1_0(0,1)\times H^1(0,1).
\end{displaymath}
The closure of $A_0$ is self-adjoint and denoted by $A$. Then
$\sigma_{\ess}(A)=\{1\}$;
see \cite[Example 2.4.11]{Tretter}. The discrete spectrum consists of the simple eigenvalue $2$ with eigenvector $(0,1)^T$, and two sequences of simple eigenvalues:
\[
\lambda_k^\pm := \frac{2+k^2\pi^2 \pm\sqrt{(k^2\pi^2 + 2)^2  - 4k^2\pi^2}}{2}\quad\text{where}\quad \lambda_k^-\nearrow1\quad\text{and}\quad\lambda_k^+\to\infty.
\]
Let $\cL_h^0$ be the FEM space of piecewise linear trial functions on $[0,1]$ with uniform mesh size $h$ and satisfying homogeneous Dirichlet boundary conditions. Let $\cL_h$ be the space without boundary conditions. Set $L_{h}=\cL_h^0\times\cL_h$. Figure 4, shows many spurious Galerkin eigenvalues in $(1,2)\subset\rho(A)$ which obscure the approximation of the genuine eigenvalue 2. We set $\mathcal{L}=L_{1/2}$ and $\Delta = [1.001,12]$, then
\[
\dim\mathcal{L}(\Delta)=2,\quad\dim \mathcal{L}=4\quad\text{and}\quad\Delta\cap\sigma(A)=\{2,\lambda^+_1\}
\]
where $\lambda_1^+\approx 10.96990625$. Evidently, we always have $(0,1)^T\in L_h$ and hence
\[
\left(
\begin{array}{c}
0\\
1
\end{array} \right)\in L_h(\Delta)\cap\mathcal{L}\quad\text{and}\quad P_{1/2}\left(
\begin{array}{c}
0\\
1
\end{array} \right) = \left(
\begin{array}{c}
0\\
1
\end{array} \right).
\]
Consequently, we always have $1\in\sigma(P_{1/2},L_h(\Delta))$ and we expect, as $h\to 0$, $\sigma(P_{1/2},L_h(\Delta))$ to converge to a second non-zero eigenvalue. Figure 5, shows this is indeed the case. The two dimensional subspace $\mathcal{M}_n$ will contain $(0,1)^T$ and an approximation of the eigenvector corresponding to $\lambda_1^+$. Therefore, $\sigma(A,\mathcal{M}_h)$ consists of $2$ and an approximation of $\lambda_1^+$. There are no spurious Galerkin eigenvalues obscuring the approximation of $\lambda_1^+$. Table 1, shows the approximation of $\lambda_1^+$ using $\sigma(A,\mathcal{M}_h)$ and $\sigma(A,\mathcal{L}_h)$. Figure 6, shows a loglog plot of the distance of $\lambda_1^+$ to  $\sigma(A,\mathcal{M}_h)$ and  $\sigma(A,\mathcal{L}_h)$. Although the convergence rates are essentially the same, the approximation by $\sigma(A,\mathcal{M}_h)$ is actually outperforming the Galerkin method.

\begin{figure}[h!]
\centering
\includegraphics[scale=.265]{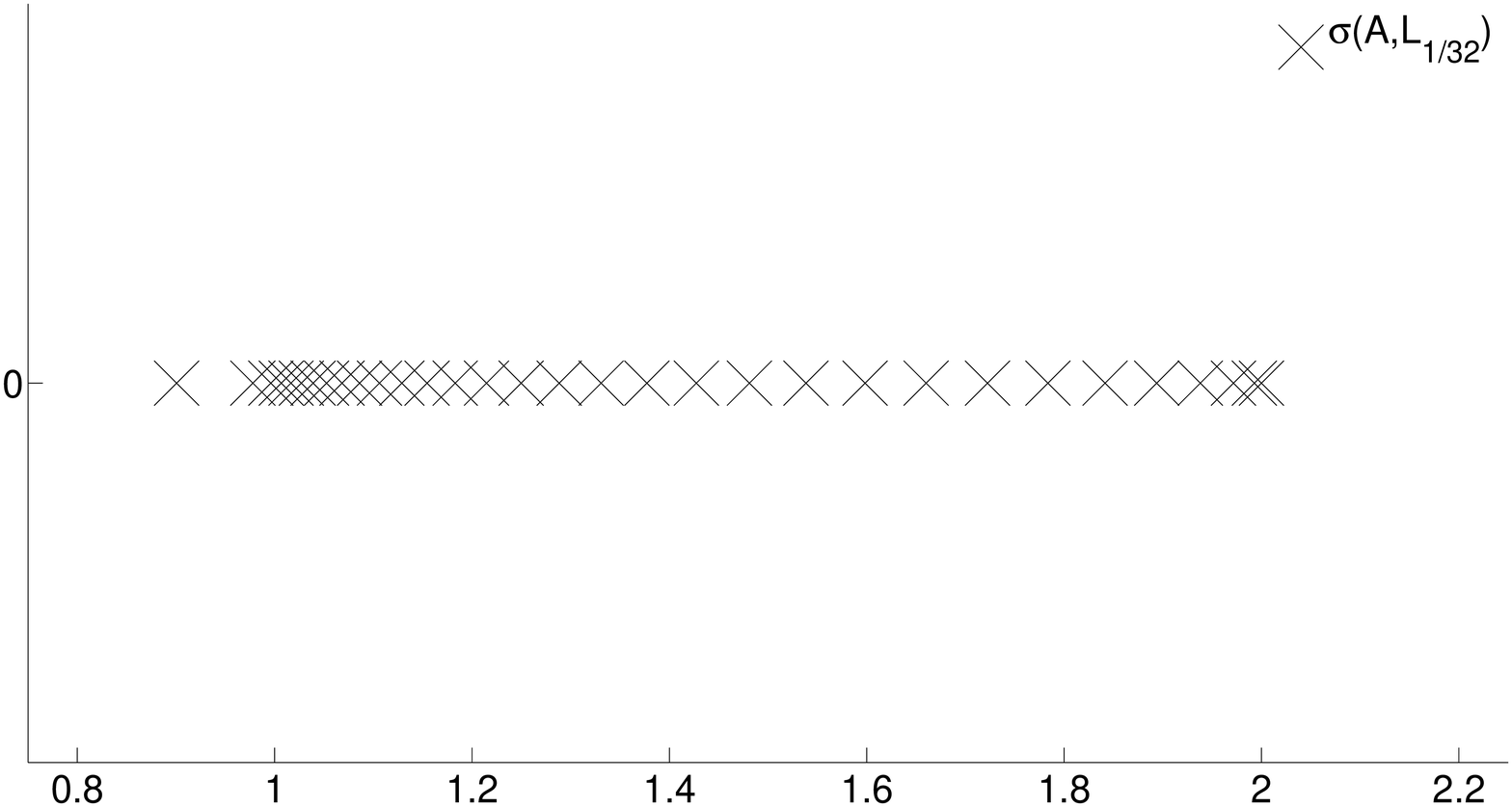}
\caption{Spurious Galerkin eigenvalues in the interval $(1,2)\subset\rho(A)$.}
\end{figure}

\begin{figure}[h!]
\centering
\includegraphics[scale=.265]{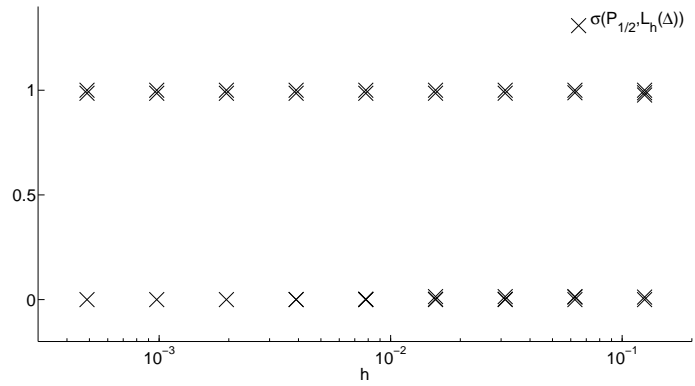}
\caption{$\sigma(P_{1/2},L_h(\Delta))$ converging to two non-zero eigenvalues.}
\end{figure}

\begin{table}[hhh]
\centering
\begin{tabular}{c|c|c}
\hspace{30pt}h\hspace{30pt} & \hspace{30pt}$\sigma(A,\mathcal{M}_h)$\hspace{30pt} & \hspace{30pt}$\sigma(A,\mathcal{L}_h)$\hspace{30pt} \\
\hline
1/8 &     11.05969611  & 11.08334840 \\
1/16 &   10.97328312 &  10.99818000 \\
1/32 &   10.97490592 &  10.97696913 \\
1/64 &   10.96960440 &  10.97167162 \\
1/128 & 10.97002620 &  10.97034757 \\
1/256 & 10.96991628 &  10.97001658\\
1/512 & 10.96991153 &  10.96993383\\
1/1024 & 10.96990927 &10.96991314
\vspace{10pt}
\end{tabular}
\caption{Approximation of $\lambda_1^+\approx 10.96990625$ using $\sigma(A,\mathcal{M}_h)$ and $\sigma(A,\mathcal{L}_h)$.}
\end{table}

\begin{figure}[h!]
\centering
\includegraphics[scale=.265]{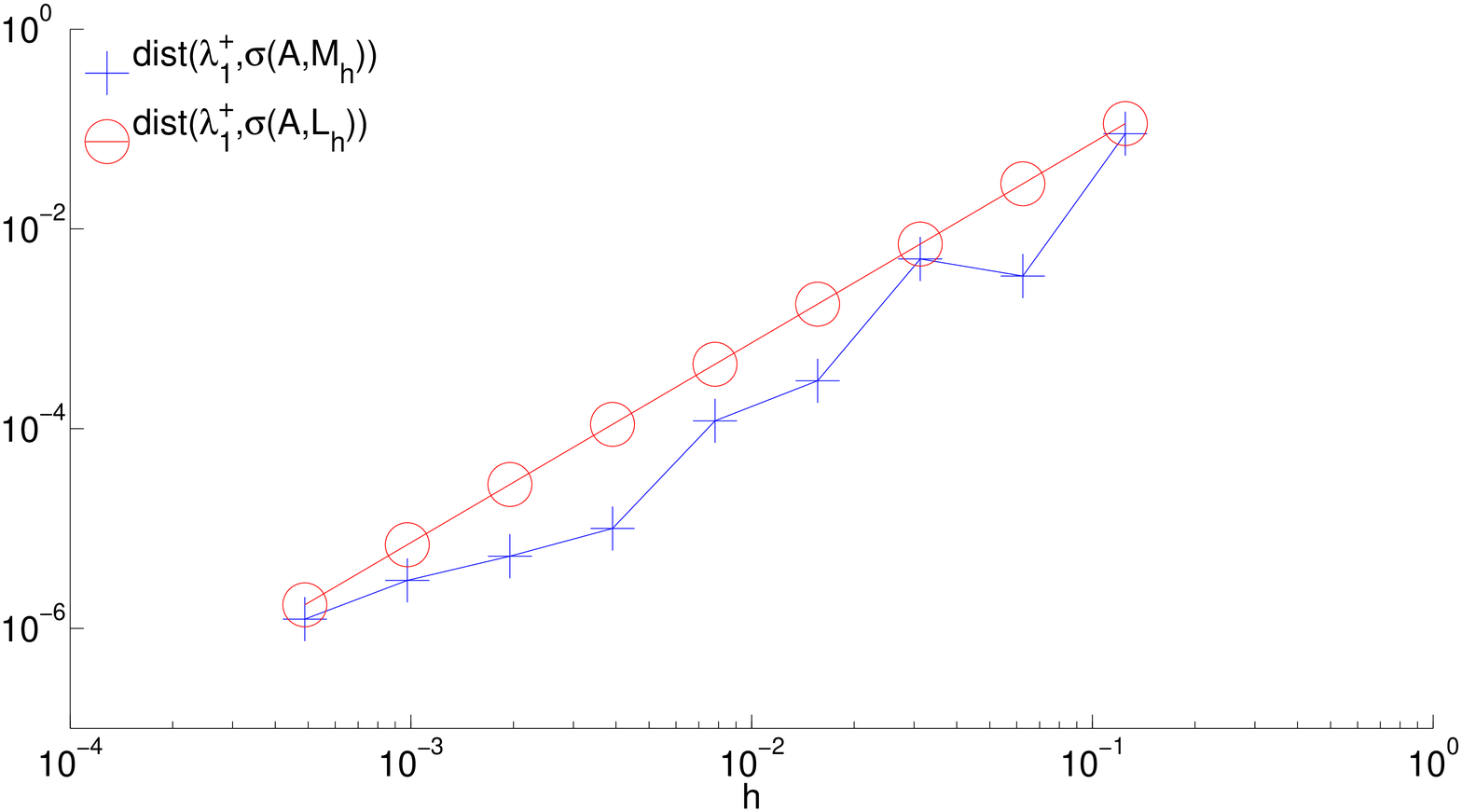}
\caption{Approximation of $\lambda_1^+$ using $\sigma(A,\mathcal{L}_h)$ and $\sigma(A,\mathcal{M}_h)$.}
\end{figure}
\end{example}

\begin{example}\label{ex3}
Let $\mathcal{H}=L^2(0,1)\times L^2(0,1)\times L^2(0,1)$ and consider the following block-operator matrix which arises in magnetohydrodynamics
\begin{displaymath}
A_0=\small{\left(\begin{array}{ccc}
        -\frac{d}{dx}(\upsilon_a^2 + \upsilon_s^2)\frac{d}{dx} + k^2\upsilon_a^2 & -i(\frac{d}{dx}(\upsilon_a^2 + \upsilon_s^2) -1)k_\perp & -i(\frac{d}{dx}\upsilon_s^2 -1)k_\parallel\\
        ~&~&~\\
	-ik_\perp((\upsilon_a^2 + \upsilon_s^2)\frac{d}{dx} +1) & k^2\upsilon_a^2 + k_\perp^2\upsilon_s^2 & k_\perp k_\parallel\upsilon_s^2\\
	~&~&~\\
	-ik_\parallel(\upsilon_s^2\frac{d}{dx} +1) & k_\perp k_\parallel\upsilon_s^2 & k_\parallel^2\upsilon_s^2
      \end{array}\right)}
\end{displaymath}
where
\[
\Dom(A_0) = H^2(0,1)\cap H^1_0(0,1)\times H^1(0,1)\times H^1(0,1).
\]
The closure of $A_0$ is self-adjoint and denoted by $A$. For simplicity we set
\[
k_\perp=k_\parallel=g=1,\quad\upsilon_{a}(x)=\sqrt{7/8 - x/2}\quad\text{and}\quad\upsilon_s(x)=\sqrt{1/8+x/2},
\]
then 
\[
\sigma_{\ess}(A)=\range(\upsilon_a^2k_\parallel)\cup\range\left(\frac{\upsilon_a^2\upsilon_s^2k_\perp}{\upsilon_a^2+\upsilon_s^2}\right) = [3/8,7/8]\cup[7/64,1/4];
\]
see \cite[Section 5]{atk} and \cite[Theorem 3.1.3]{Tretter}. The discrete spectrum contains a sequence of simple eigenvalues accumilating at $\infty$. These eigenvalues lie above, and are not close to, the essential spectrum. They are approximated by the Galerkin method, with trial spaces $L_h=\cL_{h}^0\times\cL_h\times\cL_h$, without incurring spectral pollution. 

It was shown, using the second order relative spectrum, that there is also an eigenvalue $\lambda_1\approx 0.279$ in the gap in the essential spectrum; see \cite[Example 2.7]{me2}. This eigenvalue was also located by the perturbation method; see \cite[Example 5]{me4}. In both cases, the numerical evidence suggests that the eigenvalue is simple. As shown in Figure 7, the eigenvalue $\lambda_1$ is completely obscured by spurious Galerkin eigenvalues.  We set $\mathcal{L}=L_{1/2}$ and $\Delta_1 = [1/4+0.001,3/8-0.001]$, then $\dim(\mathcal{L})=7$ and the numerical evidence, from \cite{me2,me4}, suggests that
\begin{equation}\label{sugg}
\dim(\mathcal{L}(\Delta_1))=1\quad\text{with}\quad\Delta_1\cap\sigma(A)=\{\lambda_1\}.
\end{equation}
Figure 8, shows $\sigma(P_{1/2},L_h(\Delta_1))$ which  converges to a single non-zero eigenvalue and thus provides further evidence that \eqref{sugg} is correct. The subspace $\mathcal{M}_h$ is the span of the corresponding eigenvector. Table 2, shows the approximation of $\lambda_1$ using $\sigma(A,\mathcal{M}_h)$ and the perturbation method.

\begin{figure}[h!]
\centering
\includegraphics[scale=.265]{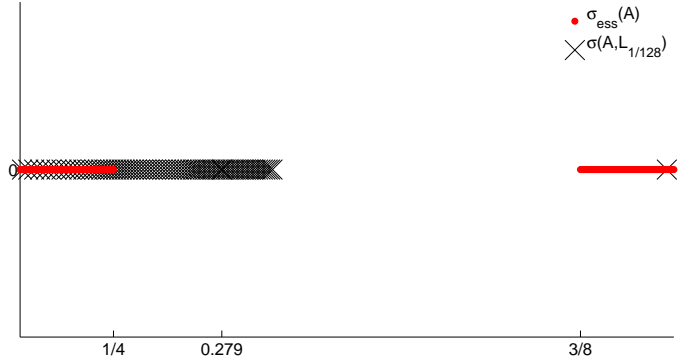}
\caption{Spurious Galerkin eigenvalues obscure the approximation of $\lambda_1\approx 0.279$.}
\end{figure}

\begin{figure}[h!]
\centering
\includegraphics[scale=.265]{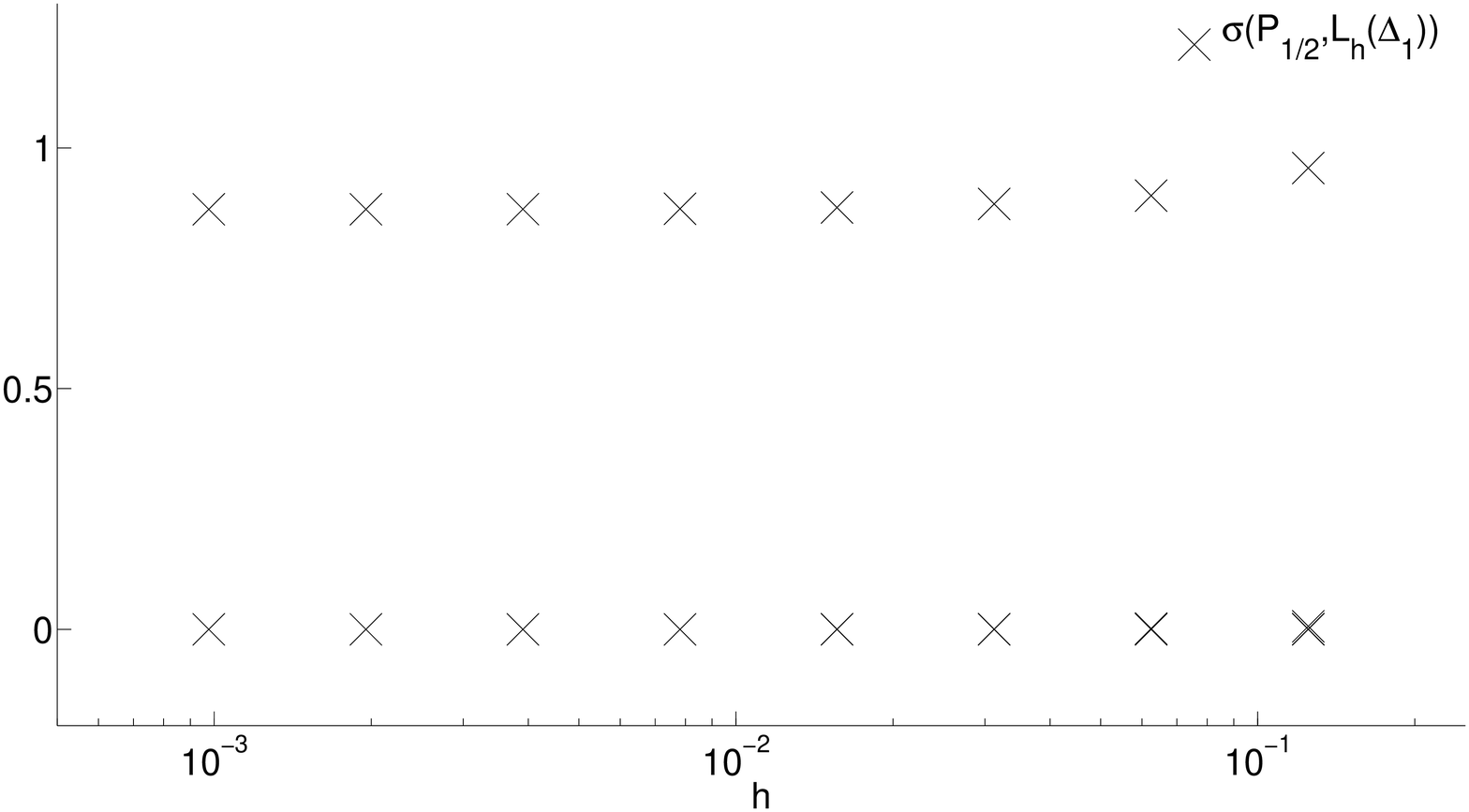}
\caption{$\sigma(P_{1/2},L_h(\Delta_1))$ converging to one non-zero eigenvalue.}
\end{figure}

\begin{table}[hhh]
\centering
\begin{tabular}{c|c|c}
\hspace{30pt}h\hspace{30pt} & \hspace{30pt}$\sigma(A,\mathcal{M}_h)$\hspace{30pt} & \hspace{30pt}$\sigma(A+iP_{2h},\mathcal{L}_h)$\hspace{30pt} \\
\hline
1/8 &     0.28037548 & 0.28071256 \\
1/16 &   0.27982165 & 0.28028198 \\
1/32 &   0.27931501 & 0.27940131 \\
1/64 &   0.27912106 & 0.27913080 \\
1/128 & 0.27905636& 0.27905757 \\
1/256 & 0.27903778 & 0.27903793\\
1/512   & 0.27903279 &0.27903281\\
1/1024 & 0.27903149 & \hspace{3pt}0.27903150
\vspace{10pt}
\end{tabular}
\caption{Approximation of $\lambda_1$ using $\sigma(A,\mathcal{M}_h)$  and  $\sigma(A+iP_{2h},\mathcal{L}_h)$.}
\end{table}

A further eigenvalue $\lambda_2\approx 1.734$ was located by the perturbation method; \cite[Example 5]{me4}. Again, the perturbation method suggests that the eigenvalue is simple.  As shown in Figure 9, the eigenvalue $\lambda_2$ is completely obscured by spurious Galerkin eigenvalues. We set $\mathcal{L}=L_{1/2}$ and $\Delta_2 = [7/8+0.001,3]$, then $\dim(\mathcal{L})=7$ and the numerical evidence, from \cite{me4}, suggests that
\begin{equation}\label{suggs}
\dim(\mathcal{L}(\Delta_2))=1\quad\text{with}\quad\Delta_2\cap\sigma(A)=\{\lambda_2\}.
\end{equation}
Figure 10, shows $\sigma(P_{1/2},L_h(\Delta_2))$ which converges to a single non-zero eigenvalue and thus provides further evidence that \eqref{suggs} is correct. The subspace $\mathcal{M}_h$ is the span of the corresponding eigenvector. Table 3, shows the approximation of $\lambda_2$ using $\sigma(A,\mathcal{M}_h)$ and the perturbation method.

\begin{figure}[h!]
\centering
\includegraphics[scale=.265]{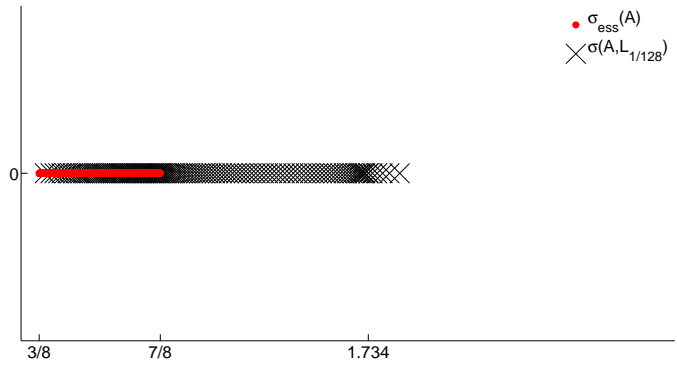}
\caption{Spectral pollution above $\sigma_{\ess}(A)$ obscures the approximation of $\lambda_2\approx 1.734$.}
\end{figure}

\begin{figure}[h!]
\centering
\includegraphics[scale=.265]{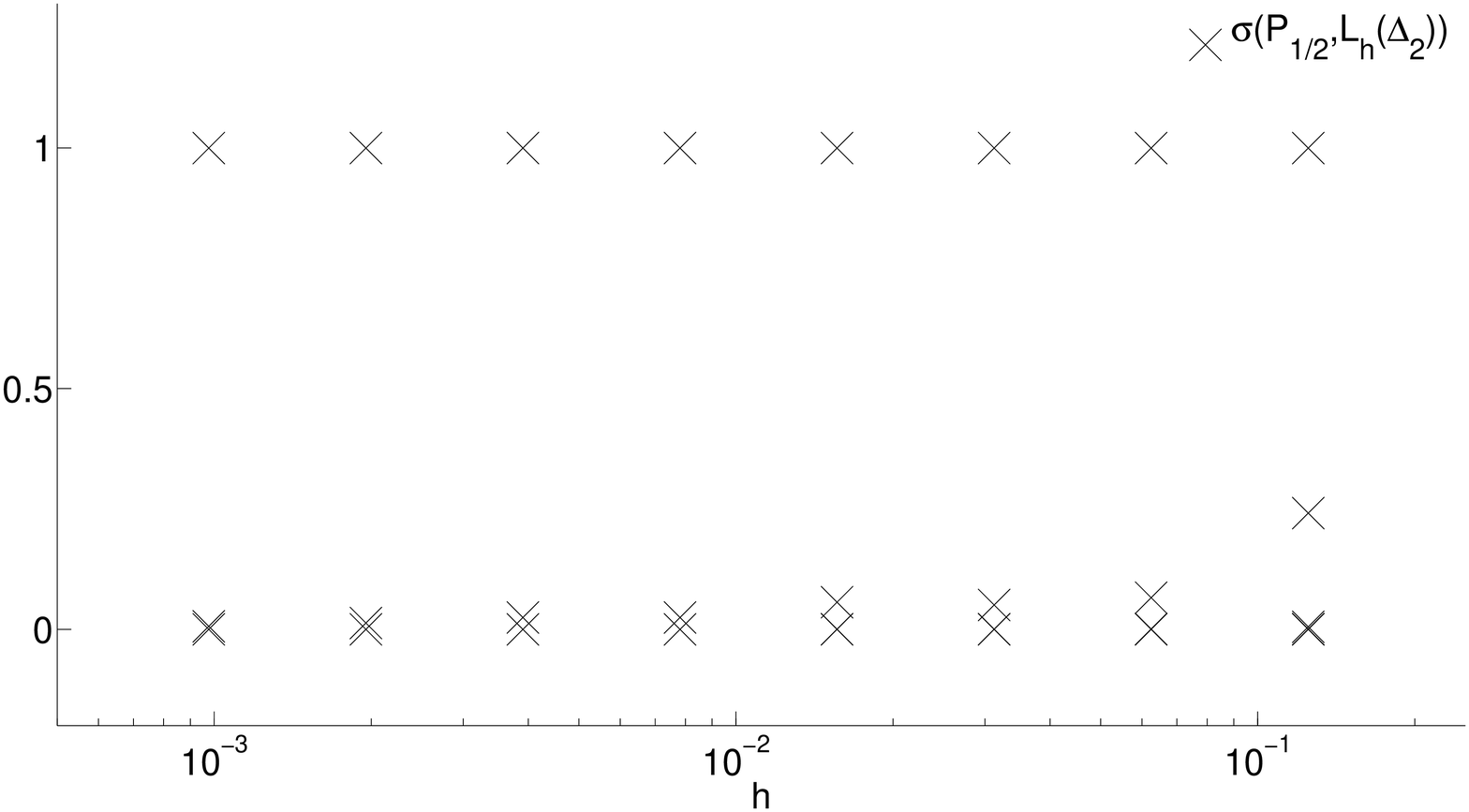}
\caption{$\sigma(P_{1/2},L_h(\Delta_2))$ converging to one non-zero eigenvalue.}
\end{figure}

\begin{table}[hhh]
\centering
\begin{tabular}{c|c|c}
\hspace{30pt}h\hspace{30pt} & \hspace{30pt}$\sigma(A,\mathcal{M}_h)$\hspace{30pt} & \hspace{30pt}$\sigma(A+iP_{2h},\mathcal{L}_h)$\hspace{30pt} \\
\hline
1/8 &     1.73461704 & 1.73467528 \\
1/16 &   1.73463871 & 1.73464550 \\
1/32 &   1.73463393 & 1.73463690 \\
1/64 &   1.73463291 & 1.73463471 \\
1/128 & 1.73463343 & 1.73463416\\
1/256 & 1.73463339 & 1.73463403\\
1/512 & 1.73463368 & 1.73463400\\
1/1025 & 1.73463384 &\hspace{3pt}1.73463399
\vspace{10pt}
\end{tabular}
\caption{Approximation of $\lambda_2$ using $\sigma(A,\mathcal{M}_h)$  and  $\sigma(A+iP_{2h},\mathcal{L}_h)$.}
\end{table}

\end{example}

\section{Conclusions}
The new technique we have presented can, in view of Corollary \ref{espaces3}, be regarded as a filter for the Galerkin method. Since if the latter does not incur spectral pollution in $\Delta$, then $\sigma(A,\mathcal{M}_n)$ will not alter the approximation. We should activate the algorithm when we have reason to be suspicious of the Galerkin method. For example, the large number of Galerkin eigenvalues just above $\sigma_{\ess}(A)$ in examples \ref{ex2} and \ref{ex3}. Also, we should always be wary of Galerkin eigenvalues in gaps in the essential spectrum, as we saw in examples \ref{ex1} and \ref{ex3}. In each case though, our algorithm easily filters out the spurious Galerkin eigenvalues and reveals an approximation of the genuine eigenvalues. We stress the ease with which this algorithm is employed; we use only the matrices required for the Galerkin method. Compared to the alternative techniques, our method is very simple, easy to apply, efficient, and accurate.

\section*{Acknowledgements}
The author is grateful to Marco Marletta for many useful discussion and acknowledges the support from the Wales Institute of Mathematical and Computational Sciences and the Leverhulme Trust grant: RPG-167.

\end{document}